\documentclass[]{article}
\usepackage{amssymb}
\usepackage{amsmath}
\usepackage{amscd}
\usepackage{mathtools}
\usepackage{amsthm}
\usepackage{color}
\usepackage[
top    = 3.00cm,
bottom = 3.00cm,
left   = 3.00cm,
right  = 2.80cm]{geometry}
\usepackage[style=alphabetic,natbib=true,bibencoding=ascii,backend=bibtex]{biblatex}
\usepackage{hyperref}
\usepackage{etoolbox}
\usepackage{tikz-cd}
\usepackage{faktor}
\usepackage{makeidx}
\usepackage{tikz-cd}
\apptocmd{\thebibliography}{\raggedright}{}{}

\usepackage{titletoc}

\dottedcontents{section}[1.5em]{\bfseries}{1.3em}{.6em}

\newtheorem{theorem}{Theorem}[section]
\newtheorem{lemma}[theorem]{Lemma}
\newtheorem{proposition}[theorem]{Proposition}
\newtheorem{corollary}[theorem]{Corollary}
\newtheorem{remark}{Remark}
\newtheorem{definition}[theorem]{Definition}
\newcommand*{\llongrightarrow}{\ensuremath{\joinrel\relbar\joinrel\relbar\joinrel\relbar\joinrel\relbar\joinrel\relbar\joinrel\relbar\joinrel\relbar\joinrel\rightarrow}}

\title{Local constancy for reductions of two-dimensional crystalline representations}
\author{Emiliano Torti \\ University of Luxembourg\\ emiliano.torti@uni.lu or torti.emiliano@gmail.com}

\begin{document}
	
	\maketitle

\begin{abstract}
We prove the existence of local constancy phenomena for reductions in a general prime power setting of two-dimensional irreducible crystalline representations. Up to twist, these representations depend on two parameters: a trace $a_p$ and a weight $k$. We find an (explicit) local constancy result with respect to $a_p$ using Fontaine's theory of $(\varphi, \Gamma)$-modules and its crystalline refinement due to Berger via Wach modules and their continuity properties. The local constancy result with respect to $k$ (for $a_p\not=0$) will follow from a local study of Colmez's rigid analytic space parametrizing trianguline representations. This work extends some results of Berger obtained in the semi-simple residual case. 
\end{abstract}

\section{Introduction}
Crystalline representations play a central role in the study of $p$-adic representations of the local absolute Galois group $G_{\mathbb{Q}_p}:= \text{Gal}(\bar{\mathbb{Q}}_p / \mathbb{Q}_p)$ (see, for example, some density results due to Berger (see Thm. IV.2.1 in \cite{Ber04}),  Chenevier (see Thm. A in \cite{Che13}), Colmez (see sec. 5.1 in \cite{Col08}), and Kisin (see Thm. 0.3 in \cite{Kis10})).
We are interested in studying irreducible crystalline representations of $G_{\mathbb{Q}_p}$ of dimension two and in particular their reductions modulo prime powers.\\
Let $p$ be an odd prime, let $k\geq2$ be an integer and $a_p \in \mathfrak{m}_\mathbb{E}$ where $\mathbb{E}$ is a finite extension of $\mathbb{Q}_p$, $\mathfrak{m}_\mathbb{E}$ denotes the maximal ideal of the ring of integers $\mathcal{O}_\mathbb{E}$ with residue field $k_\mathbb{E}$. Fix once and for all a choice of a uniformizer, say $\pi$. Let $D_{k, a_p}:=\mathbb{E} e_1 \oplus \mathbb{E} e_2$ be the filtered $\varphi$-module whose structure is given by:
$$
\varphi=\begin{pmatrix} 0 & -1 \\ p^{k-1} & a_p \end{pmatrix} \;\;\;\text{ and a filtration }
\text{Fil}^{i} (D_{k, a_p})= 
\begin{cases} 
D_{k, a_p} & \text{ if } i\leq 0 \\ 
\mathbb{E} e_1 & \text{ if } 1 \leq i \leq k-1 \\
0 & \text{ if } i \geq k
\end{cases}
$$
By a theorem of Colmez and Fontaine (see Thm. A in \cite{CF00}), there exists a unique crystalline irreducible $\mathbb{E}$-linear representation $V_{k, a_p}$ of dimension two, with Hodge-Tate weights $\{0, k-1\}$ such that $D_{\text{cris}} (V_{k, a_p}^*)=D_{k, a_p}$, where $V_{k, a_p}^*$ denotes the $\mathbb{E}$-linear dual representation of $V_{k, a_p}$. By a result of Breuil (see Prop. 3.1.1 in \cite{Breu03}), up to twist, any irreducible two-dimensional crystalline representation is isomorphic to $V_{k, a_p}$ for some $k\geq 2$ and $a_p \in \mathfrak{m}_\mathbb{E}$.\\
These results gives rise to the natural questions of whether it is possible to completely classify $V_{k, a_p}$ in terms of $k$ and $a_p$, and how $V_{k, a_p}$ varies when the parameters $k$ and $a_p$ vary $p$-adically.\\
In general, classifying the representations $V_{k, a_p}$ in characteristic zero turns out to be an hard problem even though some progress have been made in particular cases via the local Langlands correspondence\\
(e.g. see \cite{pas09}).
 Nevertheless, much progress has been made in describing the semi-simple residual reductions of the representations $V_{k, a_p}$ using different approaches. We will briefly recall the state of art in the residual case.   
Consider the $\mathbb{E}$-linear representation $V_{k, a_p}$ and denote by $T_{k, a_p}$ a $G_{\mathbb{Q}_p}$-stable lattice, we have an isomorphism $T_{k, a_p}\otimes_{\mathcal{O}_{\mathbb{E}}} \mathbb{E}\cong V_{k, a_p}$ of $G_{\mathbb{Q}_p}$-modules. Denote by $\overline{V}_{k, a_p}$ the semi-simplification of $T_{k, a_p}\otimes_{\mathcal{O}_\mathbb{E}} k_{\mathbb{E}}$; by the Brauer-Nesbitt's theorem, the representation $\overline{V}_{k, a_p}$ does not depend on the chosen $G_{\mathbb{Q}_p}$-stable lattice $T_{k, a_p}$.\\
The problem of describing the representations $\overline{V}_{k, a_p}$ has been deeply studied by many authors via the $p$-adic and mod $p$ Langlands correspondence (see for example \cite{Breu03} and  \cite{BG09}), via Fontaine's theory of $(\varphi, \Gamma)$-modules and its crystalline refinement via Wach modules (see for example \cite{BLZ04}) or via deformation theory (see for example \cite{Roz17})). However, the problem of classifying them is still open and only partial results are known (see for example \cite{BLZ04}, \cite{BG15}, \cite{BGR18}) although partial conjectures have been formulated (see Conj. 1.5 in \cite{Breu03} and Conj. 1.1 in \cite{Gha19}).\\		
In order to try to describe the reductions $\overline{V}_{k, a_p}$, one different approach consisted in finding isomorphisms between different residual representations of the form $\overline{V}_{k, a_p}$ when we let $k$ and $a_p$ vary $p$-adically. This approach has been developed by Berger et al. with the so-called local constancy results both in the trace and in the weight (see Thm. A and Thm. B in \cite{Ber12} and for for the case $a_p=0$ see Thm. 1.1.1 in \cite{BLZ04}). \\
The purpose of this article is to extend Berger's result to a prime power setting. The main difficulty lies in keeping track of the Galois stable lattices involved in the congruences because no semi-simplification process is, a priori, allowed (or defined) for general reductions modulo prime powers. Hence, in proving the existence of such congruences, a dependency on a choice of the Galois stable lattices is expected; but as we will see later, in some cases, the result will be independent of such choice. \\
The first results of the article is the following local constancy result with respect to the trace, i.e. we fix the weight and we late the trace of the crystalline Frobenius vary $p$-adically:
\begin{theorem}\emph{(Local constancy with respect to $a_p$)}\\
\label{AAA}
Let $a_p, a'_p \in \mathfrak{m}_\mathbb{E}$ and $k\geq 2$ be an integer. Let $m\in \frac{1}{e} (\mathbb{Z}_{\geq 1})$ such that $v (a_p - a'_p)\geq 2\cdot v (a_p) + \alpha(k-1)+m$, then for every $G_{\mathbb{Q}_p}$-stable lattice $T_{k, a_p}$ inside $V_{k, a_p}$ there exists a $G_{\mathbb{Q}_p}$-stable lattice $T_{k, a'_p}$ inside $V_{k, a'_p}$ such that 
$$T_{k, a_p}\otimes_{\mathcal{O}_\mathbb{E}} \mathcal{O}_\mathbb{E} / (p^m) \cong T_{k, a'_p} \otimes_{\mathcal{O}_\mathbb{E}} \mathcal{O}_\mathbb{E} / (p^m) \text{ as } G_{\mathbb{Q}_p}\text{-modules;}$$
where $\alpha(k-1)=\sum_{n\geq 1} \lfloor \frac{k-1}{p^{n-1} (p-1)}\rfloor$.
\end{theorem}
This result will be obtained by deforming p-adically the Wach module attached to the representation $T_{k, a_p}$  into a new Wach module which will correspond to the representation $T_{k, a'_p}$. The result will follow from observing that in such deformations if $a_p$ and $a'_p$ are sufficiently $p$-adically close then the corresponding Wach modules will be $p$-adically close as well; we will refer to this feature as continuity property of the Wach modules. \\
The second result of the article is the following local constancy result with respect to the weight, i.e. we fix the trace of the crystalline Frobenius and we let the weight vary $p$-adically:

\begin{theorem}\emph{(Local constancy with respect to $k$)}\\
Let $a_p\in m_\mathbb{E} -\{0\}$  for some finite extension $\mathbb{E}$ of $\mathbb{Q}_p$. Let $k\geq 2$ and $m\in\frac{1}{e} (\mathbb{Z}_{\geq1})$ be fixed. Assume that
$$k \geq (3 v(a_p)+m) \cdot \Big(1-\frac{p}{(p-1)^2}\Big)^{-1} +1.$$
There exists an integer $r=r(k, a_p)\geq 1$ such that if $k'-k\in p^{r+m} (p-1)\mathbb{Z}_{\geq 0}$ then there exist $G_{\mathbb{Q}_p}$-stable lattices $T_{k, a_p}\subset V_{k, a_p}$ and $T_{k', a_p}\subset V_{k', a_p}$ such that $$T_{k, a_p}\otimes_{\mathcal{O}_\mathbb{E}} \mathcal{O}_\mathbb{E} / (p^m) \cong T_{k', a_p} \otimes_{\mathcal{O}_\mathbb{E}} \mathcal{O}_\mathbb{E} / (p^m)\text{ as }G_{\mathbb{Q}_p}\text{-modules.}$$
\end{theorem}
The idea is to prove that the representations $V_{k, a_p}$ and $V_{k', a_p}$ are congruent modulo $p^m$ to two representations $W_k$ and $W_{k'}$ (respectively) which fit into an analytic family of trianguline representations in the sense of Berger and Colmez (see \cite{BC08}); as a consequence, the claims will follow from proving that if $k$ and $k'$ are sufficiently close in the weight space $\mathcal{W}$ then the representations $W_k$ and $W_{k'}$ are $p$-adically close as well (in a sense that will be clarified precisely later in the article). The result constitutes a converse (in the crystalline case) to a non-published theorem of Wintenberger, also proven by Berger and Colmez (see Thm 7.1.1 and Cor. 7.1.2 in \cite{BC08}) via the continuity property of the Sen periods and the Hodge-Tate weights. \\
Specializing the above theorems to the case $m=1/e$, we get a slightly stronger result then the known local constancy results in the semi-simple residual case (see Thm. A and Thm. B in \cite{Ber12}); indeed our conclusions do not involve any semi-simplification process, so for example, being residually completely reducible (i.e. direct sum of two characters modulo $p$) for some choice of lattice is also a locally constant phenomenon. \\
The motivation behind the study of local constancy phenomena modulo prime powers is two-fold. From a purely representation theoretical point of view, the interest in understanding reductions modulo prime powers of crystalline representations lies in the result of Berger on limits of crystalline representations (see \cite{Ber04}). To be more precise, Berger's result implies that if $V$ is any $p$-adic representation of $G_{\mathbb{Q}_p}$ and if $\{V_i\}_{i\in I}$ is a countable family of crystalline representations such that $T\equiv T_i \mod p^i$, where $T$ is a fixed $G_{\mathbb{Q}_p}$-stable lattice in $V$ and $T_i$ is a $G_{\mathbb{Q}_p}$-stable lattice in $V_i$, then $V$ is also crystalline.\\
Moreover, we observe that a good source of examples for the crystalline representations of the form $V_{k, a_p}$ comes from restriction at $G_{\mathbb{Q}_p}$ of Galois representations attached to classical modular forms of tame level. To be precise, let $f$ be a classical normalized cuspidal eigenform in $S_k (\Gamma_0 (N))$ where $N$ is a positive integer prime with $p$ and denote by $\rho_f$ its attached $p$-adic Galois representation constructed by Deligne and Shimura. We define by $V_p (f):={\rho_f}\big|_{G_{\mathbb{Q}_p}}$ the restriction of $\rho_f$ at the decomposition group at $p$. It is well-known (see \cite{Sch90}), that under the mild hypothesis $a_p^2 \not=4 p^{k-1}$ (see \cite{CE98}), the $G_{\mathbb{Q}_p}$-representation $V_p (f)$ is crystalline and moreover we have that $D_{\text{cris}}(V_p (f))=D_{k, a_p}$ and so $V_p (f)\cong V_{k, a_p}$ where $a_p$ is the $p$-th coefficient of the $q$-expansion of $f$. A straightforward application of the results in this article consists of using the explicit local constancy results in the trace (see Thm. \ref{AAA} and Cor. \ref{BBB}) to find upper and lower bounds for the number of non-isomorphic classes of reductions modulo prime powers of modular crystalline representations of $G_{\mathbb{Q}_p}$ coming from classical modular forms of tame level.\\
The paper is organized as follows. In Section 2, we will recall the notions of $(\varphi, \Gamma)$-module of Fontaine and of Wach module and their main properties which will be used later in the article.
In Section 3, we will recall the continuity property of Wach modules which will play a key role in the local constancy result in the trace. In Section 4, we will show how to $p$-adically deform Wach modules and we will state and prove the explicit local constancy result when the weight $k$ is fixed and we let the trace of the crystalline Frobenius $a_p$ vary. 
Finally, in Section 5, we are going to state and prove the local constancy result when we fix $a_p$ and we let the weight $k$ vary.
\bigbreak
\textbf{Acknowledgments:}  This work was prepared during a PhD at the University of Luxembourg. I would like to thank my advisor G. Wiese for the help during the writing of this article. Special thanks go to L. Berger for allowing me to visit him at the ENS in Lyon and for many enlightening conversations. \\
I would also like to thank A. Conti, L. Dembele, A. Maksoud and A. Vanhaecke for many interesting remarks.

\section{Wach modules and crystalline representations}
Let $p$ be an odd prime and let $\mathbb{E}\subseteq \overline{\mathbb{Q}}_p$ be a finite extension of $\mathbb{Q}_p$. We denote by $\mathcal{O}_\mathbb{E}$ the ring of integers of $\mathbb{E}$, by $\pi_\mathbb{E}$ a uniformizer, by $k_\mathbb{E}$ the residue field and denote by $e$ the ramification index of $\mathbb{E}$ over $\mathbb{Q}_p$. Let $\Gamma$ be a group isomorphic to $\mathbb{Z}^{\times}_p$ via a map $\chi: \Gamma \rightarrow \mathbb{Z}^{\times}_p$. Fix once and for all a topological generator of $\Gamma$ (which is procyclic as $p\not=2$), say $\gamma$. \\
Let $\mathcal{A}_{\mathbb{E}}$ be the $\pi_\mathbb{E}$-adic completion of $\mathcal{O}_\mathbb{E} [[x]][\frac{1}{x}]$.
The ring $\mathcal{A}_\mathbb{E}$ has a natural $\mathcal{O}_\mathbb{E}$-linear action of $\Gamma$ and a $\mathcal{O}_\mathbb{E}$-linear Frobenius endomorphism $\varphi$ given by the following expressions:

$$
\begin{aligned}
 &\varphi(f(x))=f((1+x)^p -1) \;\;\;\;\;\;\;\;\text{for all }f(x)\in\mathcal{A}_\mathbb{E},\\
 &\eta(f(x))=f((1+x)^{\chi(\eta)} -1) \;\;\;\;\text{for all } f(x)\in\mathcal{A}_\mathbb{E}, \forall \eta\in\Gamma.
\end{aligned}
$$
We start by recalling the following:

\begin{definition} 
An \'etale $(\varphi, \Gamma)$-module $D$ over $\mathcal{O}_\mathbb{E}$ is an $\mathcal{A}_\mathbb{E}$-module of finite type endowed with a semilinear Froebenius map $\varphi$ such that $\varphi(D)$ generates $D$ as $\mathcal{A}_\mathbb{E}$-module (this is the \'etale property) and a semilinear continuous action of $\Gamma$ which commutes with $\varphi$.\\
The category of \'etale $(\varphi, \Gamma)$-module over $\mathcal{O}_\mathbb{E}$ will be denoted by $\textbf{Mod}^{\text{\'et}}_{(\varphi, \Gamma)} (\mathcal{O}_\mathbb{E})$.
\end{definition}
Let $G_{\mathbb{Q}_p}$ be the absolute Galois group of $\mathbb{Q}_p$ and let $H_{\mathbb{Q}_p}$ be the kernel of the $p$-adic cyclotomic character $\omega: G_{\mathbb{Q}_p}\rightarrow \mathbb{Z}^{\times}_p$. Denote by $\textbf{Rep}_{\mathcal{O}_\mathbb{E}} (G_{\mathbb{Q}_p})$ the category of $\mathcal{O}_\mathbb{E}$-representations of $G_{\mathbb{Q}_p}$, i.e. the category of $\mathcal{O}_\mathbb{E}$-modules of finite type with a continuous $\mathcal{O}_\mathbb{E}$-linear action of $G_{\mathbb{Q}_p}$.\\
 By a theorem of Fontaine (see A.3.4 in \cite{Fon90}) and its generalization by Dee (see 2.2 in \cite{Dee01}) we have the following:

\begin{theorem}
\label{fontaine}
There exists an natural isomorphism
$$\mathfrak{D} : \textbf{Rep}_{\mathcal{O}_\mathbb{E}} (G_{\mathbb{Q}_p}) \rightarrow \textbf{Mod}^{\text{\'et}}_{(\varphi, \Gamma)} (\mathcal{O}_\mathbb{E})$$
given by $\mathfrak{D}(T)=(\mathcal{A}^{(\mathcal{O}_\mathbb{E})}\otimes_{\mathcal{O}_\mathbb{E}} T)^{H_{\mathbb{Q}_p}}$, where $\mathcal{A}^{(\mathcal{O}_\mathbb{E})}:=\mathcal{A}\otimes_{\mathbb{Z}_p} \mathcal{O}_\mathbb{E}$.\\
A quasi-inverse functor, which is a natural isomorphism as well, is given by:

$$\mathfrak{T} : \textbf{Mod}^{\text{\'et}}_{(\varphi, \Gamma)} (\mathcal{O}_\mathbb{E})\rightarrow \textbf{Rep}_{\mathcal{O}_\mathbb{E}} (G_{\mathbb{Q}_p})  $$
given by $\mathfrak{T}(D)=(\mathcal{A}^{(\mathcal{O}_\mathbb{E})} \otimes_{\mathcal{A}_\mathbb{E}} D)^{\varphi=1}$. 
\end{theorem}

\begin{remark} \normalfont
\label{exactness}
Note that the equivalence of categories given by the above theorem preserves the objects killed by a fixed power of a chosen uniformizer. This essentially follows from the exactness of the functor $\mathfrak{D}$ (see prop. 2.1.9 in \cite{Dee01}) and so $(\pi^n_\mathbb{E})\cdot\mathfrak{D}(T)=\mathfrak{D}((\pi^n_\mathbb{E})\cdot T)$ in the category $\textbf{Mod}^{\text{\'et}}_{(\varphi, \Gamma)} (\mathcal{O}_\mathbb{E})$. Same goes for the quasi-inverse functor $\mathfrak{T}$ (see prop. 2.1.24 in \cite{Dee01}).
\end{remark}

\begin{remark}\normalfont
\label{dual}
Let $\textbf{Mod}^{\text{\'et}, \text{tors}}_{(\varphi, \Gamma)} (\mathcal{O}_{\mathbb{E}})$ and $\textbf{Mod}^{\text{\'et}, \text{free}}_{(\varphi, \Gamma)} (\mathcal{O}_{\mathbb{E}})$ be respectively the categories of torsion and free (as $\mathcal{A}_\mathbb{E}$-modules) \'etale 
$(\varphi, \Gamma)$-modules. There is a notion of Tate's dual for such \'etale $(\varphi, \Gamma)$-modules. Let $\mathcal{B}_\mathbb{E}$ be $\mathcal{A}_\mathbb{E} [\frac{1}{p}]$ and define the Tate's dual as follows (see sec. I.2 in \cite{Col10}):
$$
\begin{aligned}
&\text{if }D\in \textbf{Mod}^{\text{\'et}, \text{tors}}_{(\varphi, \Gamma)} (\mathcal{O}_{\mathbb{E}}) \text{ then }D^{*}:=\text{Hom}_{\mathcal{A}_\mathbb{E}} \Big(D, \mathcal{B}_\mathbb{E} /\mathcal{A}_\mathbb{E} \frac{dx}{1+x}\Big)\\
&\text{if }D\in \textbf{Mod}^{\text{\'et}, \text{free}}_{(\varphi, \Gamma)} (\mathcal{O}_{\mathbb{E}}) \text{ then }D^{*}:=\text{Hom}_{\mathcal{A}_\mathbb{E}} \Big(D, \mathcal{A}_\mathbb{E} \frac{dx}{1+x}\Big)\\
\end{aligned}
$$
where $\mathcal{B}_\mathbb{E} /\mathcal{A}_\mathbb{E} \frac{dx}{1+x}$ is the inductive limit of $p^{-n} \mathcal{A}_\mathbb{E} / \mathcal{A}_\mathbb{E} \frac{dx}{1+x}$ in the category $\textbf{Mod}^{\text{\'et}, \text{tors}}_{(\varphi, \Gamma)} (\mathcal{O}_{\mathbb{E}})$ and the action of $(\varphi, \Gamma)$-module is given by:
$$\gamma\Big(\frac{dx}{1+x}\Big)=\chi(\gamma)\frac{dx}{1+x} \text{    and    }\varphi\Big(\frac{dx}{1+x}\Big)=\frac{dx}{1+x}.$$
Note the important fact that these two notions of dual are compatible in the following sense: if $D\in \textbf{Mod}^{\text{\'et}, \text{free}}_{(\varphi, \Gamma)} (\mathcal{O}_{\mathbb{E}})$ then for all $n\in\frac{1}{e}(\mathbb{Z}_{\geq 1})$ we have that $D^{*} / p^n D^{*} \cong (D/ p^n D)^{*}$ (see prop I.2.5 in \cite{Col10}).
\end{remark}

Now, we are interested in crystalline representations. These are parametrized by Wach modules as previously studied by Wach and Berger. We will start again from the definition of Wach modules and then we will go through the main properties. First, we define $\mathcal{A}^{+}_\mathbb{E}=\mathcal{O}_\mathbb{E}[[x]]$ inside $\mathcal{A}_\mathbb{E}$. It inherits naturally the actions of $\varphi$ and $\Gamma$ by restriction from $\mathcal{A}_\mathbb{E}$. Note also that $\mathcal{A}_\mathbb{E}$ is obtained by taking $\pi_\mathbb{E}$-adic completion of the localization $\mathcal{O}_\mathbb{E} [[x]] [1/x]$ (at the multiplicative set $\{1, x, x^2, \dots\}$) of $\mathcal{A}^{+}_{\mathbb{E}}=\mathcal{O}_\mathbb{E} [[x]]$. Moreover, since $\mathcal{A}_\mathbb{E}$ is Noetherian, we have that $\mathcal{A}_\mathbb{E}$ is flat as $\mathcal{A}^{+}_\mathbb{E}$-module since localization and completion preserve such property.
Following Berger (see \cite{Ber12}), we have the following

\begin{definition} 
A Wach module of height $h$ is a free $\mathcal{A}^{+}_\mathbb{E}$-module of finite rank endowed with commutative $\mathcal{A}^{+}_{\mathbb{E}}$-semilinear actions of a Frobenius map $\varphi$ and of the group $\Gamma$ such that:
$$
\begin{aligned}
(1)\;\;\;& D(N):=\mathcal{A}_\mathbb{E} \otimes_{\mathcal{A}^{+}_\mathbb{E}} N \;\in\; \textbf{Mod}^{\text{\'et}}_{(\varphi, \Gamma)} (\mathcal{O}_\mathbb{E}),\\
(2)\;\;\;&\Gamma \text{ acts trivially on }N/xN,\\
(3)\;\;\;&N/\varphi^{*}(N) \text{ is killed by }Q^h,
\end{aligned}
$$
where $\varphi^{*}(N)$ denotes the $\mathcal{A}^{+}_\mathbb{E}$-module generated by $\varphi(N)$, and $Q=\frac{(1+x)^p-1}{x} \in\mathcal{A}^{+}_\mathbb{E}$.
\end{definition}
We recall that the Wach modules are the right objects to specialize Fontaine's equivalence to crystalline representations. Indeed, we have the following (see Prop. 1.1 in \cite{Ber12}):

\begin{proposition} 
Let $N$ be a Wach module of height $h$, then the $\mathbb{E}$-linear representation $\mathfrak{V}(N):=\mathbb{E} \otimes_{\mathcal{O}_\mathbb{E}} \mathfrak{T}(\mathcal{A}_\mathbb{E} \otimes_{\mathcal{A}^{+}_\mathbb{E}} N)$ of $G_{\mathbb{Q}_p}$ is crystalline with Hodge-Tate weights in the interval $[-h; 0]$; and $D_{\text{cris}}(\mathfrak{V}(N))\cong \mathbb{E}\otimes_{\mathcal{O}_\mathbb{E}} N/xN$ as $\varphi$-modules. Moreover, all crystalline representations with Hodge-Tate weights in $[-h;0]$ arise in this way. 
\end{proposition}

\section{Continuity of the Wach modules}

Berger proved (see sec. III.4 in \cite{Ber04}) that the there exists an equivalence of categories between Wach modules (over $\mathcal{A}^{+}_\mathbb{E}$) and $G_{\mathbb{Q}_p}$-stable lattices of crystalline representations. Denote by $\mathfrak{N}$ the functor that associates to each $G_{\mathbb{Q}_p}$-stable lattice of a crystalline representation its Wach module. In this section we will study a bit more in the specific the arithmetic relation between a Wach module and its corresponding $\mathcal{O}_\mathbb{E}$-linear representation which is a $G_{\mathbb{Q}_p}$-stable lattice of a crystalline representation. In particular, we will prove that, in some natural sense, $p$-adically close Wach modules will correspond to $p$-adically close $\mathcal{O}_\mathbb{E}$-linear representations and viceversa.
We will start by clarifying what we mean by $p$-adically close Wach modules. Given two Wach modules $N_1$ and $N_2$ we say that $N_1$ and $N_2$ are congruent modulo some prime power, i.e. 
$N_1\equiv N_2 \mod \pi^m$ for some $m\in\mathbb{Z}_{\geq 1}$, if there exists a $\mathcal{A}^{+}_\mathbb{E}$-module isomorphism between $N_1 \otimes_{\mathcal{A}^{+}_\mathbb{E}} \mathcal{A}^{+}_\mathbb{E} / (\pi_{\mathbb{E}}^m)$ and $N_2 \otimes_{\mathcal{A}^{+}_\mathbb{E}} \mathcal{A}^{+}_\mathbb{E} / (\pi_{\mathbb{E}}^m)$ which is $(\varphi, \Gamma)$-equivariant. \\
Note that, essentially by definition, we have that $N_1 \equiv N_2 \mod \pi_{\mathbb{E}}^m$ if and only if there exist basis of $N_1$ and $N_2$ as $\mathcal{A}^{+}_\mathbb{E}$-modules such that, after defining $P_1=\text{Mat}(\varphi |_{N_1})$, $P_2=\text{Mat}(\varphi |_{N_2})$, $G_1=\text{Mat}(\gamma |_{N_1})$, $G_2=\text{Mat}(\gamma |_{N_2})$  (note that $P_1, P_2, G_1, G_2 \in \text{M}_{d\times d} (\mathcal{A}^{+}_\mathbb{E})$ where $d=\text{rank}_{\mathcal{A}^+_\mathbb{E}}(N_i)$ for $i=1, 2$) it follows that:

$$
\begin{cases}P_1\equiv P_2 \mod \pi_{\mathbb{E}}^m,\\ G_1 \equiv G_2 \mod \pi_{\mathbb{E}}^m .\end{cases}
$$
We have the following continuity result (this is Thm. IV.1.1 in \cite{Ber04}):

\begin{proposition} \label{necessarylocalconstancy}
Let $T_1$ and $T_2$ two Galois stable lattices inside two crystalline $\mathbb{E}$-linear representation $V_1$ and $V_2$ of Hodge-Tate weights inside $[-r; 0]$, and assume there is an $n\in\frac{1}{e} (\mathbb{Z}_{\geq 1})$ and $n\geq \alpha(r)$ such that $T_1 \otimes_{\mathcal{O}_\mathbb{E}} \mathcal{O}_\mathbb{E}/ (p^n)\cong T_2  \otimes_{\mathcal{O}_\mathbb{E}} \mathcal{O}_\mathbb{E} / (p^n)$ as $G_{\mathbb{Q}_p}$-modules, then $\mathfrak{N}(T_1)\equiv \mathfrak{N}(T_2) \mod p^{n-\alpha(r)}$.
\end{proposition}

Now, if $N$ is a Wach module denote by $\mathfrak{T}(N):=\mathfrak{V}(D(N))$ the $\mathcal{O}_\mathbb{E}$-linear representation of $G_{\mathbb{Q}_p}$ attached to $N$. We recall that $\mathfrak{T}(N)\otimes \mathbb{E}$ is a crystalline representation. \\
We are interested in the following result:

\begin{proposition} Let $N_1$ and $N_2$ be two Wach modules (over $\mathcal{O}_\mathbb{E}$) with the same rank as $\mathcal{A}^{+}_\mathbb{E}$-modules. Assume that $N_1\equiv N_2 \text{ mod } \pi_{\mathbb{E}}^n$ for some $n\in\mathbb{Z}_{\geq 1}$. Then $\mathfrak{T}(N_1)\otimes_{\mathcal{O}_\mathbb{E}} \mathcal{O}_\mathbb{E} /(\pi_{\mathbb{E}}^n)\cong\mathfrak{T}(N_2)\otimes_{\mathcal{O}_\mathbb{E}} \mathcal{O}_\mathbb{E} /(\pi_{\mathbb{E}}^n)$ as $G_{\mathbb{Q}_p}$-modules.

\end{proposition}

\begin{proof}
Consider the \'etale $(\varphi, \Gamma)$-modules $D(N_i)=N_i \otimes_{\mathcal{A}^+_\mathbb{E}} \mathcal{A}_\mathbb{E}$ for $i=1, 2$. 
Since $\mathcal{A}_\mathbb{E}$ is flat as $\mathcal{A}^{+}_\mathbb{E}$-module (module structure given by inclusion) we have the following chain of isomorphism of torsion \'etale $(\varphi, \Gamma)$-modules:

$$D(N_1)/\pi_{\mathbb{E}}^n D(N_1)\cong N_1/\pi_{\mathbb{E}}^n N_1 \otimes_{\mathcal{A}^{+}_\mathbb{E}} \mathcal{A}_\mathbb{E} \cong N_2/\pi_{\mathbb{E}}^n N_2 \otimes_{\mathcal{A}^{+}_\mathbb{E}} \mathcal{A}_\mathbb{E}\cong D(N_2)/\pi_{\mathbb{E}}^n D(N_2).$$
Now, the claim follows just by applying Fontaine's functor $\mathfrak{T}$, which is exact (see Theorem \ref{fontaine} and see Remark \ref{exactness}).
\end{proof}
In the following sections, the idea will be to apply the above continuity results of the Wach modules to study the local constancy phenomena for reductions modulo prime powers of $G_{\mathbb{Q}_p}$-stable lattices inside $\mathbb{E}$-linear crystalline representations of dimension two. 

\section{Local constancy with respect to the trace}

In this section we are going to prove an explicit local constancy result with respect to the trace (i.e. the weight $k$ will be fixed) for reductions modulo prime powers of representations of the type $V_{k, a_p}$.

\subsection{Some linear algebra of Wach modules}

As explained in the previous section, a congruence between Wach modules (modulo some prime power) can be translated into a congruence (modulo the same prime power) between systems of matrices representing the $(\varphi, \Gamma)$ actions on the Wach modules involved. 
In this section, we will see how to $p$-adically deform a Wach module into another one via linear algebra means for the systems of matrices associated to the $(\varphi, \Gamma)$-module structure.\\
As long as it will be possible, we will keep the same notation as Berger (see \cite{Ber12}).
We recall that $p$ is an odd prime and $\mathbb{E}$ is a finite extension of $\mathbb{Q}_p$ with ramification index $e$. Let $v$ be the normalized $p$-adic valuation (i.e. $ v(p)=1$). Let $r\geq 1$ be a integer and define $\alpha(r):= \sum_{j= 1}^{r} v(1-\chi(\gamma)^{j} )$ (see \cite{Ber12}). The constant $\alpha(r)$ has also an explicit description given by $\alpha(r)=\sum_{n\geq 1} \lfloor \frac{r}{ p^{n-1} (p-1)}\rfloor$. \\
We start by recalling two useful results in linear algebra (see Lemma 2.1 in \cite{Ber12}):

\begin{lemma}
If $P_0\in M_2 (\mathcal{O}_\mathbb{E})$ is a matrix with eigenvalues $\lambda \not=\mu$, and if $\delta=\lambda -\mu$, then there exists $Y\in M_2 (\mathcal{O}_\mathbb{E})$ such that $Y^{-1} \in \delta^{-1} M_2 (\mathcal{O}_\mathbb{E})$ and $Y^{-1} P_0 Y=\begin{pmatrix} \lambda & 0 \\ 0 &\mu\end{pmatrix}$. 
\end{lemma}
and the following corollary (see cor. 2.2 in \cite{Ber12}):

\begin{corollary}
If $\alpha\in \frac{1}{e} \mathbb{Z}_{\geq 0}$ and $\epsilon\in\mathcal{O}_\mathbb{E}$ are such that $v(\epsilon)\geq 2 v(\delta)+\alpha$, then there exists $H_0 \in p^{\alpha} M_2 (\mathcal{O}_\mathbb{E})$ such that $\text{det} (\text{Id}+H_0)=1$ and $\text{Tr}(H_0 P_0)=\epsilon$. 
\end{corollary}

The corollary above represents the starting point to deform a Wach module into another one. Given the matrix $P_0$, it gives a $p$-adically small matrix $H_0$ such that the product $H_0 P_0$ will have a prescribed $p$-adically small trace.
In practice, this will be applied when $P_0$ is obtained by the action of $\varphi$ on $D_{\text{cris}} (V_{k, a_p})$ for some $k\geq 2$ and $a_p \in\mathfrak{m}_{\mathbb{E}}$.\\
The idea behind the next results is to show how $H_0$ gives rise to a deformation of a whole system of matrices (attached to a Wach module) to a $p$-adically close one preserving the characterizing linear algebra properties of the action on $\varphi$ and $\Gamma$ on the Wach module. \\
We have the following result (it is a little generalization of Prop. 2.3 in \cite{Ber12}):

\begin{proposition}
Let $m\in \frac{1}{e} \mathbb{Z}_{\geq 0}$ and $d\in\mathbb{Z}_{\geq 2}$. If $G\in \text{Id}+x M_d (\mathcal{O}_\mathbb{E} [[x]])$ and $k\geq 2$ and $H_0 \in p^{\alpha (k-1)+m} M_d (\mathcal{O}_\mathbb{E})$, then there exists $H\in p^{m} M_d (\mathcal{O}_\mathbb{E}[[x]])$ such that $H(0)=H_0$ and $HG\equiv G\gamma(H) \text{ mod }x^k$.
\end{proposition}

\begin{proof} 
As $G\in \text{Id}+x M_d (\mathcal{O}_\mathbb{E} [[x]])$, we can write $G=\text{Id}+x G_1+ x^2 G_2 + \dots$ where $G_i \in \text{M}_d (\mathcal{O}_\mathbb{E})$ for all $i\in\mathbb{Z}_{\geq 1}$. We prove that for any positive integer $r$, there exists an $H_r \in p^{\alpha(k-1)-\alpha(r)+m} M_d (\mathcal{O}_\mathbb{E})$ such that if we define $H=H_0+x H_1+x^2 H_2+\dots +x^{k-1} H_{k-1}$ we have that $HG\equiv G \gamma(H) \text{ mod }x^k$.\\
We start from $r=1$, then since $\gamma(H)=H_0+\gamma(x) H_1 + \gamma(x^2) H_2+ \dots \gamma(x^{k-1}) H_{k-1}$ and for all $w\in \mathbb{Z}_{\geq 1}$ we have $\gamma(x^w)=((1+x)^{\chi(\gamma)}-1)^{w}$, we deduce that we can define $H_1$ such that $(1-\chi(\gamma))H_1=G_1 H_0-H_0 G_1$. Since by hypothesis $H_0 \equiv 0 \text{ mod }p^{\alpha(k-1)+m}$ and by definition $\alpha(1)=v_p (1-\chi(\gamma))$, we deduce that $H_1 \in p^{\alpha(k-1)-\alpha(1)+m} M_d (\mathcal{O}_\mathbb{E})$.\\
Using now the same argument, one can actually see how $H_r$ is uniquely determined by $H_0 , H_1 , \dots H_{r-1} $ and moreover $(1-\chi(\gamma)^r) H_r \in p^{\alpha(k-1)-\alpha(r-1)+m} M_d (\mathcal{O}_\mathbb{E})$, note that $\alpha (r)=\alpha(k-1)-\alpha(r-1)$. To be precise, we prove this by induction on $r\leq k-1$. The first case $r=1$ is proven above, now assume the case $r-1$, we are going to prove the statement for $r$.\\
It is straightforward to prove that, expanding the expression $HG\equiv G\gamma(H) \mod x^k$, the following identity holds:
$$
(1-\chi(\gamma)^r) H_r=\sum^{r-1}_{h=0} \Big( \sum^{h}_{n=0} \gamma(x^n)_{h} H_n\Big) G_{r-h}-\sum_{i=0}^{r-1} H_i G_{r-i}, \quad\quad\text{for }r\leq k-1
$$
where $\gamma(x^n)_{h}$ is the $h$-th coefficient (i.e. coefficient of $x^h$) of the polynomial $\gamma(x^n)$. Note now that the map $\alpha(n)$ is non-decreasing as $n$ grows. Hence, by inductive hypothesis, we deduce that for any $i$ such that $0\leq i \leq r-1$ we have $H_i\equiv 0 \mod (p^{\alpha(k-1)-\alpha(r-1)+m})$. Since $v(1-\chi(\gamma)^r)+\alpha(r-1)=\alpha(r)$, we deduce that $H_r\equiv 0 \mod (p^{\alpha(k)-\alpha(r)+m})$. This concludes the proof. 
\end{proof}

The following result completes the linear algebra deformation process on a system of matrices that will represent the $(\varphi, \Gamma)$-action on Wach modules:
\begin{proposition}
Let $m\in\frac{1}{e} \mathbb{Z}_{\geq 0}$ and $d\in\mathbb{Z}_{\geq 2}$. Let $G\in \text{Id}+x M_d (\mathcal{O}_\mathbb{E} [[x]])$ and $P\in M_d (\mathcal{O}_\mathbb{E} [[x]])$ satisfy $P\varphi(G)=G\gamma (P)$ and $\text{det}(P)=Q^{k-1}$ where $Q=\frac{(1+x)^p-1}{x}$.\\
If $H_0 \in p^{\alpha (k-1)+m} M_d (\mathcal{O}_\mathbb{E})$, then there exist $G' \in \text{Id}+x M_d (\mathcal{O}_\mathbb{E} [[x]])$ and $H\in p^m M_d (\mathcal{O}_\mathbb{E} [[x]])$ such that:

$$
\begin{aligned}
&(1)\;\;\;H(0)=H_0;\\ 
&(2)\;\;\;P' \varphi(G')=G' \gamma (P') \text{, where }\;P'=(\text{Id}+H)P;\\
&(3)\;\;\;P\equiv P' \text{ mod }p^m ;\\
&(4)\;\;\;G\equiv G' \text{ mod }p^m .\\
\end{aligned}
$$
\end{proposition}

\begin{proof} After applying the previous proposition for the existence of the matrix $H$ which satisfies $H(0)=H_0$, the existence of $G'$ follows directly from Prop. 2.4 in \cite{Ber12}. Hence the claims (1) and (2) hold.\\
The claim (3) is clear since $H\in p^m M_d (\mathcal{O}_\mathbb{E} [[x]])$ implies that $P\equiv P' \text{ mod }p^m$. What it is left to prove is (4), i.e. $G\equiv G' \text{ mod }p^m$. In order to prove this, we need to look at how the matrix $G'$ is defined. \\
The matrix $G'$ is constructed as an $x$-adic limit inside $\text{Id}+x\text{M}_d (\mathcal{O}_\mathbb{E} [[x]])$ (note that we are dealing with non-commutative rings). Define $G'_k:=G$ and observe that it satisfies by construction $G'_k-P' \varphi(G'_k) \gamma (P')^{-1}=x^k R_k$ for some $R_k \in M_d (\mathcal{O}_\mathbb{E} [[x]])$. Then define $G'$ as the $x$-adic limit of $G'_j$, for $j\geq k$, which satisfies $G'_{j+1}=G'_j+x^j S_j$ for some $S_j \in M_d (\mathcal{O}_\mathbb{E})$, and $G'_j-P' \varphi(G'_j) \gamma(P')^{-1}=x^j R_j$ where $R_j \in M_d (\mathcal{O}_\mathbb{E} [[x]])$.\\
We will prove by induction that:

$$
\begin{aligned}
(1)&\;\; R_j \equiv 0 \text{ mod }p^m,\\
(2)&\;\; G'_j\equiv G \text{ mod }p^m.
\end{aligned}
$$
First the case $j=k$: since $H\equiv 0 \text{ mod }p^m$ then $P\equiv P' \text{ mod }p^m$ which implies that $R_k \equiv 0 \text{ mod }p^m$, and by construction $G'_k=G$ so the first case of induction is done.
Now assume $j\geq k$ and that the above claims hold for $j$, we will prove them for $j+1$.\\
We have that there exists $S_j \in M_d (\mathcal{O}_\mathbb{E})$ such that:
$$
\begin{aligned}
G'_{j+1}-P' \varphi(G'_{j+1}) \gamma(P')^{-1}&=G'_j+x^j S_j-P' \varphi(G'_j)\gamma(P')^{-1}-P' x^j Q^j S_j \gamma(P')^{-1}=\\
&=x^j (R_j+S_j-Q^{j-k+1}P' S_j Q^{k-1}) \gamma(P')^{-1}) \in x^{j+1} M_d (\mathcal{O}_\mathbb{E} [[x]]).
\end{aligned}
$$
Note that we used that $\varphi$ acts trivially on $\text{M}_d (\mathcal{O}_\mathbb{E})$ and that $\varphi(x^j)=x^j Q^j$ for all $j\in\mathbb{Z}_{\geq 1}$.
We want to prove that there exists $S_j \in p^m M_d (\mathcal{O}_\mathbb{E})$ such that:
$$
R_j+S_j-Q^{j-k+1} P' S_j Q^{k-1}\gamma(P')^{-1} \in x M_d (\mathcal{O}_\mathbb{E}).
$$
Evaluating the above expression at $x=0$, the claim is equivalent to prove that there exists $S_j \in p^m M_d (\mathcal{O}_\mathbb{E})$ such that:
$$
S_j -p^{j-k+1} P' (0) S_j p^{k-1} (\gamma (P')^{-1})(0)=-R_j (0).
$$
Now, since $R_j \equiv 0 \text{ mod }p^m$, we have that $R_j(0)\equiv 0 \text{ mod }p^m$. It is clear that the map $S\mapsto S-p^{j-k+1} P'(0) S p^{k-1}(\gamma(P')^{-1})(0)$ gives a bijection of $M_d (\mathcal{O}_\mathbb{E})$. Moreover, it is also clear that it is a bijection of $p^m M_d (\mathcal{O}_\mathbb{E})$. As $R_j(0)\equiv 0 \text{ mod }p^m$, we have the existence of $S_j \equiv 0 \text{ mod }p^m$ such that the above relations are satisfied. By inductive hypothesis $R_j \equiv 0 \text{ mod }p^m$, so $R_{j+1} \equiv 0 \text{ mod }p^m$. Since $S_j \equiv 0 \text{ mod } p^m$ implies that $G'_{j+1}=G'_j+x^j S_j \equiv G \text{ mod }p^m$. This concludes the proof.  
\end{proof}

\subsection{Local constancy modulo prime powers with respect to $a_p$}
Here we will apply the continuity properties of the Wach modules to prove local constancy results modulo prime powers. In this section we will focus on fixing the weight and letting the trace vary.
As in the previous sections, let $k\geq 2$ be a positive integer and let $a_p \in \mathfrak{m}_{\mathbb{E}}$.\\
The main result of this section is the following (this is a generalization of Thm. A of \cite{Ber12}):

\begin{theorem} \label{theoremA}
Let $a_p, a'_p \in \mathfrak{m}_\mathbb{E}$ and $k\geq 2$ be an integer. Let $m\in \frac{1}{e} (\mathbb{Z}_{\geq 1})$ such that $v (a_p - a'_p)\geq 2\cdot v (a_p) + \alpha(k-1)+m$, then for every $G_{\mathbb{Q}_p}$-stable lattice $T_{k, a_p}$ inside $V_{k, a_p}$ there exists a $G_{\mathbb{Q}_p}$-stable lattice $T_{k, a'_p}$ inside $V_{k, a'_p}$ such that 
$$T_{k, a_p}\otimes_{\mathcal{O}_\mathbb{E}} \mathcal{O}_\mathbb{E} / (p^m) \cong T_{k, a'_p} \otimes_{\mathcal{O}_\mathbb{E}} \mathcal{O}_\mathbb{E} / (p^m) \text{ as } G_{\mathbb{Q}_p}\text{-modules.}$$
\end{theorem}

\begin{proof}
Consider the $G_{\mathbb{Q}_p}$-representation $V^{*}_{k, a_p}=\text{Hom}_{\mathbb{E}} (V_{k,a_p}, \mathbb{E})$; it is crystalline and it has Hodge-Tate weights $0$ and $-(k-1)$. 
Let $T_{k, a_p}^{*} :=\text{Hom}_{\mathcal{O}_\mathbb{E}} (T_{k,a_p}, \mathcal{O}_\mathbb{E})$ be the $G_{\mathbb{Q}_p}$-stable lattice in $V_{k, a_p}^{*}$, dual of $T_{k, a_p}$. By a result of Berger (see prop. III.4.2 and III.4.4 in \cite{Ber04}), it is possible to attach to $T_{k, a_p}^{*}$ a Wach module $N_{k, a_p}$ of height $k-1$. Fixing a basis of $N_{k, a_p}$ as $\mathcal{A}_\mathbb{E}^{+}$-module (we recall that in our notation $\mathcal{A}_{\mathbb{E}}^{+}=\mathcal{O}_\mathbb{E} [[x]]$), the actions of $\varphi$ and $\gamma$ on $N_{k, a_p}$ can be respectively represented by the matrices $P\in \text{Mat}_2 (\mathcal{A}_\mathbb{E}^{+})$ and $Q\in \text{Id}+x\text{Mat}_2 (\mathcal{A}_\mathbb{E}^{+})$. Note that since the actions of $\varphi$ and $\Gamma$ commute (in a semi-linear sense) we have that $P\varphi(G)=G\gamma(P)$. We recall also that the matrix $P(0)$ has characteristic polynomial $T^2-a_p T+p^{k-1}$. Now, let $a'_p \in\mathcal{O}_{\mathbb{E}}$ be as in the hypothesis, i.e. it satisfies $v(a_p-a'_p)\geq 2\cdot v(a_p)+\alpha(k-1)+m$ for some $m\in \frac{1}{e} (\mathbb{Z}_{\geq 1})$.
Applying in sequence the results in section 4.1, we deduce the existence of two matrices $P'$ and $G'$ which give rise to a Wach module $N'$ and such that $P\equiv P' \mod p^m$ and $G\equiv G' \mod p^m$. 
Since by construction $P'=(\text{Id}+H)P$, we have that evaluating at $x=0$ we deduce that the characteristic polynomial of $P'(0)$ is $T^2-a'_p T+p^{k-1}$ (note that $\text{Trace}(H(0) P(0))=a'_p-a_p$ and $\text{Det}(\text{Id}+H(0))=1$). By a result of Berger (see prop. 1.2 in \cite{Ber12}), we can deduce that $N'=N_{k, a'_p}$, or equivalently $V_{k, a'_p}=\mathfrak{T}(N')\otimes_{\mathcal{O}_\mathbb{E}} \mathbb{E}$ and $D_{\text{cris}}(V_{k, a'_p}^{*})=N'/xN' \otimes_{\mathcal{O}_\mathbb{E}}\mathbb{E}$.\\
Since $P\equiv P' \mod p^m$ and $G\equiv G' \mod p^m$, we have that $N_{k, a_p}\equiv N_{k, a'_p} \mod p^m$. As a consequence of Remark \ref{exactness}, we have that $\mathfrak{D}(N_k, a_p)\equiv \mathfrak{D}(N_{k, a'_p}) \mod p^m$, i.e.

$$\mathfrak{D}(N_k, a_p)\otimes_{\mathcal{O}_\mathbb{E}} \mathcal{O}_\mathbb{E} /(p^m) \cong \mathfrak{D}(N_{k, a'_p})\otimes_{\mathcal{O}_\mathbb{E}} \mathcal{O}_\mathbb{E} / (p^m) \text{ in the category } \textbf{Mod}^{\text{\'et}}_{(\varphi, \Gamma)} (\mathcal{O}_{\mathbb{E}}).$$
Hence, we define $T_{k, a'_p}:= \mathfrak{T} (\mathfrak{D}(N_{k, a'_p})^{*})$ which is a $G_{\mathbb{Q}_p}$-stable lattice in $V_{k, a'_p}$ that satisfies (since Fontaine's functor $\mathfrak{T}$ is compatible with duals):

$$T_{k, a_p}\otimes_{\mathcal{O}_\mathbb{E}} \mathcal{O}_{\mathbb{E}}/(p^m) \cong T_{k, a'_p}\otimes_{\mathcal{O}_\mathbb{E}} \mathcal{O}_\mathbb{E} / (p^m) \text{ as } G_{\mathbb{Q}_p} \text{-modules.}$$
Indeed, since $\mathfrak{D}(N_{k, a_p})\equiv \mathfrak{D}(N_{k, a'_p}) \mod p^m$, by exactness of Fontaine's functor $\mathfrak{T}$ we can deduce that $T_{k, a_p}^{*} \equiv \mathfrak{T}(\mathfrak{D}(N_{k, a'_p})) \mod p^m$.
This completes the proof of the theorem.
\end{proof}



\subsection{Converse of local constancy with respect to the trace}

Via the continuity properties of the Wach modules, it is also possible to find an explicit necessary condition for the existence of local constancy phenomena modulo prime powers. In precise terms, let $k\geq2$ be an integer and let $a_p ,a'_p \in \mathfrak{m}_{\mathbb{E}}$; then we have the following:

\begin{proposition}
\label{BBB}
Let $m\in\frac{1}{e}(\mathbb{Z}_{\geq 1})$ and assume $m\geq \alpha(k-1)$. If $V_{k, a_p} \equiv V_{k, a'_p} \mod p^m$, then $v(a_p - a'_p)\geq m-\alpha(k-1)$.
\end{proposition}

\begin{proof}
This is a straightforward application of Berger's Proposition \ref{necessarylocalconstancy}. Indeed, we have that $D_{\text{cris}} (V^{*}_{k, a_p})=N_{k, a_p}/x N_{k, a_p}\otimes \mathbb{E}$ and $D_{\text{cris}} (V^{*}_{k, a'_p})=N_{k, a'_p}/x N_{k, a'_p}\otimes \mathbb{E}$ for some Wach modules $N_{k, a_p}$ and $N_{k, a'_p}$ corresponding respectively to the $G_{\mathbb{Q}_p}$-stable lattices in $V_{k, a_p}$ and $V_{k, a'_p}$ that are congruent modulo $p^m$. By proposition 3.1, we have that $N_{k, a_p}\equiv N_{k, a'_p} \mod p^{m-\alpha(k-1)}$ and looking at the characteristic polynomials of $\varphi$ acting on $N_{k, a_p}/x N_{k, a_p}$ and $N_{k, a'_p}/x N_{k, a_p}$ the claim follows.\end{proof}

\section{Local constancy with respect to the weight}
In this section, we are going to prove a local constancy result for reductions modulo prime powers once we fix the trace of the crystalline Frobenius $a_p$ and we let the weight $k$ vary. \\
In order to simplify the notation, we will say that two $\mathbb{E}$-linear representations $V$ and $V'$ of $G_{\mathbb{Q}_p}:=\text{Gal}(\overline{\mathbb{Q}}_p / \mathbb{Q}_p)$ are congruent modulo some prime power (i.e. $V\equiv V' \mod \pi^n$ for some $n\in\mathbb{Z}_{\geq 1}$) if there exist $G_{\mathbb{Q}_p}$-stable lattices $T\subset V$ and $T' \subset V'$ such that we have an isomorphism $$T\otimes_{\mathcal{O}_\mathbb{E}} \mathcal{O}_\mathbb{E} / (\pi^n) \cong T'\otimes_{\mathcal{O}_\mathbb{E}} \mathcal{O}_\mathbb{E} / (\pi^n) \;\;\text{ of }G_{\mathbb{Q}_p} \text{-modules.}$$
Note that the above definition requires a bit of attention when used as it clearly doesn't define an equivalence relation (in general, it is not symmetric nor transitive). 
The main result of this section is the following:

\begin{theorem}
\label{weight}
Let $p$ be an odd prime. Let $a_p\in m_\mathbb{E} -\{0\}$ for some finite extension $\mathbb{E}/\mathbb{Q}_p$. Let $k\geq 2$ be an integer and $m\in\frac{1}{e} (\mathbb{Z}_{\geq1})$ be fixed. Assume that
$$k \geq (3 v(a_p) +m)\cdot \Big(1-\frac{p}{(p-1)^2}\Big)^{-1} +1. \quad\quad\quad\quad\quad (*)$$
There exists an integer $r=r(k, a_p)\geq 1$ such that if $k'-k\in p^{r+m} (p-1)\mathbb{Z}_{\geq 0}$ then $V_{k, a_p}\equiv V_{k', a_p} \mod p^m$ 

\end{theorem}

\begin{remark} \normalfont 
 As it will be clear from the proof below, the condition in the hypothesis is not optimal in the sense that it can be replaced by the weaker condition given by the system:
 
$$
\begin{cases} 
k\geq3v(a_p)+\alpha(k-1)+1+m, \\
k'\geq3v(a_p)+\alpha(k'-1)+1+m.
\end{cases}
$$
as in Berger's result (see Thm. B in \cite{Ber12}).
For the sake of simplicity, we just assume a stronger condition which has the advantage that it is explicit in the weight, doesn't depend on the function $\alpha$ and automatically holds for $k'$ if it holds for $k$ assuming $k'\geq k$.\\
The condition in the theorem can be deduced directly from the above conditions by noticing that $\alpha(k-1)=\sum_{n\geq 1} \big\lfloor\frac{k-1}{p^{n-1} (p-1)} \big\rfloor$ satisfies the inequality $\alpha(k-1)\leq \frac{(k-1)p}{(p-1)^2}.$
\end{remark}

\begin{remark}\normalfont
Note that Theorem \ref{weight} and Theorem \ref{theoremA} can be applied in sequence (i.e. one can first deform the weight and then deform the trace) in order to have a local constancy result in which both the trace and the weight vary. Note that the order in which the theorems can be applied in sequence cannot be switched, as it is always necessary to keep track of the lattices involved in the congruences. 
\end{remark}

\begin{remark}\normalfont
It could be interesting to consider the question of finding explicitly a radius for the local constancy in the weight. Partial results have been obtained by Bhattacharya (see \cite{Bha18}). As already pointed out in the introduction, the above theorem can be seen as a converse (in a special crystalline case) of a non-published theorem of Winterberger, proven by Berger and Colmez as a consequence of a continuity property of the Sen periods and the Hodge-Tate weights (see \cite{BC08}). Their result provides a connection between the local constancy radius of our theorem and the constant $c(2, \mathbb{Q}_p)$ of the result of Berger and Colmez. To be precise, combining theorem \ref{weight} above and cor. 7.1.2 in \cite{BC08} we get the upper bound on the radius for the local constancy in weight $p^{-(r+m)}\leq p^{-(\lfloor \frac{m}{2}\rfloor -c(2, \mathbb{Q}_p))}$.
\end{remark}
In order to prove the theorem, the idea is to make use of Kedlaya's theory of $(\varphi,\Gamma)$-modules of slope zero over the Robba ring (see \cite{Ked04}) and to realize the representations $V_{k, a_p}$ (for $k$ suff. big) as trianguline representations in the sense of Colmez (see \cite{Col08}). A theorem of Colmez will then ensure us that locally such representations vary in a continuous way. We will make this precise in the next section.
We refer the reader to \cite{Ber11} for a nice summary on the theory of trianguline representations and its applications in arithmetic geometry. \\
Let $\mathcal{R}_\mathbb{E}$ be the Robba ring with coefficients in $\mathbb{E}$ and for any multiplicative character $\delta:\mathbb{Q}_p^{\times} \rightarrow \mathbb{E}^\times$, denote by $\mathcal{R}_{\mathbb{E}} (\delta):=\mathcal{R}_\mathbb{E} e_{\delta}$ the $(\varphi, \Gamma)$-module (in the sense of Kedlaya, see \cite{Ked04}) of rank one obtained by defining the actions $\varphi(e_{\delta})=\delta(p) e_{\delta}$ and $\gamma(e_\delta)=\delta(\chi(\gamma)) e_\delta$ for all $\gamma \in\Gamma$, where $\chi$ denotes the chosen fixed isomorphism between $\Gamma$ and $\mathbb{Z}_p^{\times}$.\\
Colmez (see Thm. 0.2 in \cite{Col08}) proved that all $(\varphi, \Gamma)$-modules of rank one arise as $\mathcal{R}_\mathbb{E} (\delta)$ for a unique multiplicative character $\delta$; moreover, if $\delta_1$, $\delta_2 : \mathbb{Q}_p^\times \rightarrow \mathbb{E}^\times$ are multiplicative characters then $\text{Ext}^1 (\mathcal{R}_\mathbb{E} (\delta_1), \mathcal{R}_\mathbb{E} (\delta_2))$ is an $\mathbb{E}$-vector space of dimension 1 unless 
$\delta_1 \delta_2^{-1}$ is of the form $x^{-i}$ for some integer $i\geq0$, or $|x|x^i$ for some integer $i\geq 1$; in both cases, the dimension over $\mathbb{E}$ is two and the attached projective space is isomorphic to $\mathbb{P}^1 (\mathbb{E})$; here $x$ denotes the identity character of $\mathbb{Q}_p^{\times}$.\\
Hence, where the extension is not unique (up to isomorphism), one will need to specify the corresponding parameter in $\mathbb{P}^1 (\mathbb{E})$ usually called $L$-invariant and denoted as $\mathfrak{L}$. 
The corresponding Galois representation will be denoted by $V(\delta_1, \delta_2, \mathfrak{L})$. For an extensive discussion about $\mathfrak{L}$-invariant, we refer the reader to the original article of Colmez (see sec. 4.5 in \cite{Col08}).\\
Each trianguline representation $V(\delta_1, \delta_2)$ corresponds (up to considering blow-up in case $\delta_1 \delta_2^{-1} =x^{-i}$ or $|x|x^i$; see \cite{Col08}) to the point $(\delta_1, \delta_2)\in \mathfrak{X}\times \mathfrak{X}$ where $\mathfrak{X}$ is isomorphic (non-canonically, as there are choices involved) to the $\mathbb{Q}_p$-rigid analytic space $\mu(\mathbb{Q}_p)\times \mathbb{G}_m^{\text{rig}}\times \mathbb{B}^1 (1, 1)^{-}_{\mathbb{Q}_p}$.\\ From now on, we denote by $\mathbb{B}^1 (a, r)^{+}_{\mathbb{Q}_p}$ the closed affinoid rigid $\mathbb{Q}_p$-ball centered in $a$ and with radius $r$. The expression $\mathbb{B}^1 (a, r)^{-}_{\mathbb{Q}_p}$ will instead denote the open rigid $\mathbb{Q}_p$-ball as a $\mathbb{Q}_p$-rigid analytic space which parametrizes multiplicative characters $\delta: \mathbb{Q}_p^{\times} \rightarrow \mathbb{L}^\times$ where $\mathbb{L}$ is a finite extension of $\mathbb{Q}_p$.

\subsection{Proof of the theorem}
Let $k'$ be an integer satisfying $k'-k \in(p-1)\mathbb{Z}_{\geq 0}$. The claim is to prove that if $k'$ and $k$ are sufficiently $p$-adically close then the corresponding representations are isomorphic modulo a prescribed prime power. \\
The assumption $(*)$ on the weight $k$ allow us, applying Theorem \ref{theoremA}, to deduce that:

$$
\begin{aligned}
V_{k, a_p+\frac{p^{k-1}}{a_p}} &\equiv V_{k, a_p} \mod p^m , \\
 V_{k', a_p+\frac{p^{k'-1}}{a_p}}  &\equiv V_{k', a_p} \mod p^m .
\end{aligned}
$$
Indeed, note that assumption $(*)$ implies that $k-1>2v(a_p)$ and hence in this specific case, the Theorem  \ref{theoremA} can be applied both starting from $V_{k, a_p}$ or starting from $V_{k,  a_p+\frac{p^{k-1}}{a_p}}$ (same goes for $k'$) and hence this gives us a strong control over the lattices involved in the congruences. 
Therefore, this first step reduces the claim to prove that if $k'$ and $k$ are sufficiently $p$-adically close (in the weight space) then we have the congruence $V_{k, a_p+\frac{p^{k-1}}{a_p}} \equiv V_{k', a_p+\frac{p^{k'-1}}{a_p}} \mod p^m$.\\
The following proposition of Colmez (see prop. 3.1 in \cite{Ber12} or see sec. 4.5 in \cite{Col08}) allow us to realize the above representations as trianguline representations:

\begin{proposition}
If $z\in\mathfrak{m}_\mathbb{E}$ is a root of $z^2-a_p z+p^{k-1}$ which satisfies $v(z)<k-1$, then we have that $V(\mu_z , \mu_{\frac{1}{z}} \chi^{1-k}, \infty)=V_{k, a_p}^{*}$.\\
Here $\mu_z :\mathbb{Q}_p^{\times}\rightarrow \mathbb{E}^{\times}$ is a character such that $\mu_z (p)=z$ and $\mu_z (\mathbb{Z}_p^{\times})=1$ and $\chi :\mathbb{Q}_p^{\times}\rightarrow \mathbb{E}^{\times}$ is a character such that $\chi(p)=1$ and $\chi(y)=y$ for all $y\in\mathbb{Z}_p^{\times}$.
\end{proposition}
Hence, if $k'-1\geq k -1> v(a_p)$, we have that the crystalline representations $V_{k, a_p+\frac{p^{k-1}}{a_p}}$ and $V_{k', a_p+\frac{p^{k'-1}}{a_p}} $ coincide respectively with the trianguline representations $V(\mu_{a_p}, \delta_{k, a_p}, \infty)$ and $V(\mu_{a_p}, \delta_{k', a_p}, \infty)$, where $\delta_{k, a_p}:=\mu_{\frac{1}{a_p}} \chi^{1-k}$ and $\delta_{k', a_p}=\mu_{\frac{1}{a_p}} \chi^{1-k'}$.
Since the $\mathfrak{L}$-invariant is going to be $\infty$ for all the trianguline representations involved, we will drop the notation $V(\cdot, \cdot, \infty)$ writing simply $V(\cdot, \cdot)$. \\
The trianguline representations $V(\mu_{a_p}, \delta_{k, a_p})$ and $V(\mu_{a_p}, \delta_{k', a_p})$ define two $\mathbb{E}$-points, respectively $u_{k, a_p}=(\mu_{a_p}, \delta_{k, a_p})$ and $u_{k', a_p}=(\mu_{a_p}, \delta_{k', a_p})$, on the rigid analytic space $\mathfrak{X}\times \mathfrak{X}$ (see \cite{Col08}) parametrizing couples of multiplicative characters of $\mathbb{Q}_p^{\times}$ with values in $\mathbb{L}^{\times}$ where $\mathbb{L}$ is some finite extension of $\mathbb{Q}_p$.\\
The following lemma represents the first step for constructing a family of trianguline representations interpolating $V(\mu_{a_p}, \delta_{k, a_p})$ and $V(\mu_{a_p}, \delta_{k', a_p})$ when $k$ and $k'$ will be sufficiently $p$-adically close (in the weight space):
\begin{lemma} \label{morphism}
Let $\alpha\in 1+p\mathbb{Z}_p$, then we have that 
$$
\begin{aligned}
\psi_{\alpha} : \mathbb{B}^{1} (0,1)_{\mathbb{Q}_p}^{+} &\longrightarrow \mathbb{B}^{1} (1,|\alpha-1|)_{\mathbb{Q}_p}^{+}\\
[s] &\mapsto [\text{exp}_p (s \cdot \text{log}_p (\alpha))]
\end{aligned}
$$
is an isomorphism in the category of $\mathbb{Q}_p$-rigid analytic spaces. Here $[s]$ denotes the maximal ideal of $\mathbb{Q}_p \langle T \rangle$ corresponding to the element $s\in\overline{\mathbb{Z}}_p$ and the analogue for $\psi_\alpha ([s])$.
\end{lemma}
\begin{proof} 
First, we will clarify that $\psi_\alpha$ is a well defined map for every $\alpha\in 1+p\mathbb{Z}_p$. For all $s\in \overline{\mathbb{Z}}_p$, we have that $\psi_{\alpha}$ converges when evaluated in $s$ since $|s\cdot \text{log}_p (\alpha)|\leq |\alpha-1|\leq p^{-1}$. Moreover since $\psi_\alpha (s) \in \mathbb{Q}_p (s)$, we have that the map $\psi_{\alpha}$ is Galois equivariant, i.e. $\psi_{\alpha} (\sigma (s))=\sigma (\psi_{\alpha} (s))$ for every $\sigma \in G_{\mathbb{Q}_p}$. Note also that we can find the explicit expression $\psi_\alpha (s)=\alpha^s=(1+(\alpha-1))^s=\sum_{n\geq 0} {{s}\choose{n}} (\alpha-1)^n$ which converges for every $s\in\overline{\mathbb{Z}}_p$. This allow us to define $\psi_\alpha$ on the set of $G_{\mathbb{Q}_p}$-orbits of $\overline{\mathbb{Z}}_p$ which can be identified set-theoretically with $\mathbb{B}^1 (0, 1)^{+}_{\mathbb{Q}_p}$.
Proving that $\psi_\alpha$ is a morphism of $\mathbb{Q}_p$-rigid analytic spaces (affinoid in this case) boils down to show that the induced map on the corresponding affinoid algebras is a morphism. \\
Let $\mathcal{O}_{0,1}$ and $\mathcal{O}_{1, |\alpha-1|}$ denote respectively the $\mathbb{Q}_p$-affinoid algebras attached to the $\mathbb{Q}_p$-affinoid spaces $\mathbb{B}^{1} (0,1)_{\mathbb{Q}_p}^{+}$ and $\mathbb{B}^{1} (1,|\alpha-1|)_{\mathbb{Q}_p}^{+} $. \\
We have that the associated map $\psi_{\alpha}^* : \mathcal{O}_{1,|\alpha-1|} \rightarrow \mathcal{O}_{0, 1}$ is given by $\psi_{\alpha}^* (f)=f \circ \psi_{\alpha}$ for all $f\in \mathcal{O}_{1, |\alpha-1|}\cong \mathbb{Q}_p \langle \frac{T-1}{\alpha-1}\rangle$.\\
In order to show that $\psi_{\alpha}^*$ is a morphism of affinoid algebras, since it is given by the pull-back, it is sufficient to show that it is a well-defined map, in the sense that  $\psi_{\alpha}^* (f)$ belongs to $\mathcal{O}_{0,1}\cong \mathbb{Q}_p \langle T \rangle$; or in other words, it is a converging series. Indeed, the problem is that, in general, composition of $p$-adic analytic functions is not analytic. The convergence property will be deduced by the following convergence criterion (see Thm. 4.3.3 in \cite{Gou97}): 
\begin{theorem}
Let $f(X)=\sum a_n X^n$ and $g(X)=\sum b_n X^n$ be formal power series in $\overline{\mathbb{Q}}_p [[X]]$ with $g(0)=0$, and let $h(X)=f(g(X))$ be their formal composition.\\ 
Suppose that:\\
(i) $g(x)$ converges,\\
(ii) $f(g(x))$ converges (i.e. plugging the number to which $g(x)$ converges into $f(X)$ gives a convergent series)\\
(iii) for every $n\in \mathbb{Z}_{\geq 0}$, we have $|b_n x^n|\leq |g(x)|$ (i.e. no term of the series converging to $g(x)$ is bigger than the sum).
Then $h(x)$ also converges, and $f(g(x))=h(x)$.
 \end{theorem}
Indeed, in our case, it is enough to prove that for all $n\in\mathbb{Z}_{\geq 0}$ we have 
$$|c_n s^n|\leq |\psi_{\alpha} (s)| \;\;\;\;\; \text{ where } \psi_\alpha (s) =\text{exp}_p (s\cdot \text{log}_p (\alpha))=\sum_{n\geq 0} c_n s^n \;\;\;\text{    and   } \;\;c_n:=\frac{(\text{log}_p (\alpha))^n}{n!}.$$
We have that $ |\psi_{\alpha} (s)|=|\text{exp}_p (s\cdot \text{log}_p (\alpha))|=1$ for all $s\in \mathbb{B}^1 (0, 1)^{+}$, hence since $|s|\leq 1$, it is sufficient to prove that $|c_n|\leq 1$.\\
Since $\alpha \in 1+p\mathbb{Z}_p$ and since the $p$-adic logarithm is an isometry we have that $|\text{log}_p (\alpha)|^n = |\alpha -1|^n \leq p^{-n}$. We also recall from classical $p$-adic analysis that $v_p (n!)< \frac{n}{p-1}$ or in other words $|n!| > p^{-\frac{n}{p-1}}$. It follows at once that
$$|c_n|=\frac{|\text{log}_p (\alpha)|^n}{|n!|}<\frac{p^{-n}}{p^{-\frac{n}{p-1}}}<1.$$
Note that the inverse $\psi^{-1}_{\alpha} :  \mathbb{B}^{1} (1,|\alpha-1|)_{\mathbb{Q}_p}^{+} \longrightarrow  \mathbb{B}^{1} (0,1)_{\mathbb{Q}_p}^{+}$ sends $[t]$ to $\Big[\frac{\text{log}_p (t)}{\text{log}_p (\alpha)}\Big]$ and it is, via the same argument, a well-defined morphism of $\mathbb{Q}_p$-affinoid spaces.
This concludes the proof.
\end{proof}
We will make use of the map $\psi_\alpha$ just defined to construct a family of points (i.e. trianguline representations) on $\mathfrak{X}^2$ which will pass through $u_{1-k}$ (i.e. the representation $V(\mu_{a_p}, \mu_{\frac{1}{a_p}} \chi^{1-k})$). \\
For each $s\in\overline{\mathbb{Z}}_p$, we define a multiplicative character of $\mathbb{Q}_p^{\times}$ as follows:

$$
\begin{aligned}
\delta_{k, a_p}^{(s)} : \mathbb{Q}_p^{\times} &\longrightarrow \mathbb{E}(s)^{\times} \\
x&\longmapsto \delta_{k, a_p}^{(s)} (x):= \mu_{\frac{1}{a_p}}(x)\cdot \omega(x)^{1-k}\cdot \psi_{\langle x \rangle }(s),
\end{aligned}
$$
where $\mathbb{E}(s)$ is the finite extension obtained from $\mathbb{E}$ by adding $s$; and where $x=p^{v_p (x)} \omega(x) \langle x \rangle$ is the unique decomposition given by a fixed isomorphism $\mathbb{Q}_p^{\times}\cong p^{\mathbb{Z}}\times \mu(\mathbb{Q}_p) \times 1+p\mathbb{Z}_p$. \\
Note that $\psi_{\langle x \rangle}(s)$ is an element in $\mathbb{Q}_p (s)$ since ${{s}\choose {n}} (\langle x \rangle-1)^n\in\mathbb{Q}_p (s)$ for any $n\in\mathbb{Z}_{\geq 1}$.\\
Finally, we are ready to apply this in the context of rigid analytic spaces, indeed we will define a 1-dimensional $p$-adic family of points in $\mathfrak{X}^2$ through which we will control the ``$p$-adic'' distance between $u_{1-k}$ and $u_{1-k'}$.\\ 
We define $\mathcal{Z}$ to be the $\mathbb{Q}_p$-affinoid spaces given by $\{\mu_{a_p}\}\times\{1/a_p \}\times \{\zeta_p^{k-1}\} \times \mathbb{B}^1 (0,1)^{+}_{\mathbb{Q}_p}$; here $\{\mu_{a_p}\}$ denotes the singleton corresponding to the character $\mu_{a_p}$ on $\mathfrak{X}$, the singleton $\{1/a_p \}$ corresponds to the $\mathbb{E}$-point $1/a_p$ in $\mathbb{G}_m^{\text{rig}}$ and the singleton $\{ \zeta_p^{k-1}  \}$ corresponds to the point $\zeta_{p-1}^{k-1}$ in $\mu(\mathbb{Q}_p)$.\\
By a little abuse of notation, we will still denote a point in $\mathcal{Z}$ by $s$ for the corresponding point $s\in \mathbb{B}^1 (0,1)^{+}_{\mathbb{Q}_p}$. Now, we define the injective map:

$$
\begin{aligned}
\Phi: \mathcal{Z} &\longrightarrow \mathfrak{X}^2\\
s&\longmapsto \Phi(s):=(\mu_{a_p}, \delta_{k, a_p}^{(s)})
\end{aligned}
$$
and note that if $k'\in\mathbb{Z}_{\geq 2}$ satisfies $k'-k \in (p-1)\mathbb{Z}_{\geq 0}$ we have, by construction, that $\Phi(1-k)=u_{k, a_p}$ and $\Phi(1-k')=u_{k', a_p}$ since $\delta_{k, a_p}^{(1-k)}=\mu_{\frac{1}{a_p}} \chi^{1-k}$ and $\delta_{k, a_p}^{(1-k')}=\mu_{\frac{1}{a_p}} \chi^{1-k'}$.\\

\begin{proposition}
The map $\Phi : \mathcal{Z} \rightarrow \mathfrak{X}^2$ is a rigid analytic closed immersion. 
\end{proposition}

\begin{proof}
In order to see that $\Phi$ is a morphism of $\mathbb{Q}_p$-rigid analytic spaces it is sufficient to observe that, decomposing $\mathfrak{X}^2$ as $\mathfrak{X}\times \mu(\mathbb{Q}_p)\times \mathbb{G}^{\text{rig}}_m \times \mathbb{B}^1 (1, 1)_{\mathbb{Q}_p}^{-}$, the map $\Phi$ is a product of constant morphisms and $\psi_{1+p}$:

$$
\begin{aligned}
\Phi: \mathcal{Z} &    \llongrightarrow \mathfrak{X}^2                                   =\mathfrak{X}\times \mu(\mathbb{Q}_p)\times \mathbb{G}_m^{\text{rig}} \times \mathbb{B}^1 (1,1)_{\mathbb{Q}_p}^{-} \\
s                         &\longmapsto (\mu_{a_p}, \delta_{k, a_p}^{(s)})                =(\mu_{a_p}, \delta_{k, a_p}^{(s)} (\zeta_{p-1}), \delta_{k, a_p}^{(s)} (p), \delta_{k, a_p}^{(s)} (1+p))\\
                        &                                                                                   \qquad\qquad\qquad\;\;\;=(\mu_{a_p}, [\zeta_{p-1}^{1-k}], [a_p^{-1}], \psi_{1+p}(s)).&\;
\end{aligned}
$$
The universal property of fiber products (in the category of rigid analytic spaces) allows us to reduce the claim to prove that the composition of $\Phi$ with the projection on the last factor of $\mathfrak{X}^2$, which is exactly $\psi_{1+p}$, belongs to $\text{Mor}(\mathcal{Z}, \mathbb{B}^1 (1, 1)_{\mathbb{Q}_p}^{-})$. This follows at once from Lemma \ref{morphism}. 
Moreover, the image is an affinoid subdomain of $\mathfrak{X}^2$ making $\Phi$ a closed immersion.
\end{proof}
Now, the heart of the proof is that the representations attached to points of $\mathfrak{X}^2$ vary locally in a continuous way. In precise terms, this is the following result of Colmez and Chenevier (see Prop. 5.2 in \cite{Col08} and its generalization Prop. 3.9 in \cite{Che13}, see also Prop. 3.2 in \cite{Ber12}):

\begin{theorem}
\label{continuitya}
Let $\delta_1, \delta_2 :\mathbb{Q}_p^{\times} \rightarrow \mathbb{E}^{\times}$ be two characters such that $\delta_1 \delta_2^{-1} \not=x^{-i}$ for some $i\geq 0$, where $x$ denotes the identity character of $\mathbb{Q}_p^{\times}$. Let $u=(\delta_1, \delta_2)$ be the corresponding point in $\mathfrak{X}^2$. Then there exists a open neighborhood $\mathfrak{U}$ of $u$ and a finite, free $\mathcal{O}_\mathfrak{U}$-module $\mathbb{V}$ of rank 2 with an action of $G_{\mathbb{Q}_p}:=\text{Gal}(\bar{\mathbb{Q}}_p, \mathbb{Q}_p)$ such that $\mathbb{V}(\tilde{u})=V(\delta_1 (\tilde{u}), \delta_2 (\tilde{u}))$ for every $\tilde{u}\in\mathfrak{U}$.
\end{theorem}
As we are interested in points inside $\mathfrak{U}\subset \mathfrak{X}^2$ which is an open neighborhood of $u_{1-k}$, we will first prove that if $k'$ is sufficiently $p$-adically close to $k$ (close as points in the weight space) then also $\Phi(1-k')=u_{1-k'}$ will lie in $\mathfrak{U}$.\\
Without loss of generality, as $\mathfrak{U}$ is an admissible open we can assume (up to restriction) that $(\mathfrak{U}, \mathcal{O}_{\mathfrak{U}})$ is an affinoid space. Since $\Phi$ is a morphism of rigid analytic spaces, it is in particular continuous for the $G$-topology, hence $\Phi^{-1} (\mathfrak{U})$ is an admissible open of the affinoid space $\mathcal{Z}$.\\
In particular, we can deduce that there exists a minimal $r\in\frac{1}{e} (\mathbb{Z}_{\geq1})$ such that the affinoid subdomain 
$\mathcal{Z}_r :=\{\mu_{a_p}\}\times \{[a^{-1}_p]\}\times\{[\zeta_p^{k-1}]\}\times\mathbb{B}^{1} (1-k, p^{-r})^{+}_{\mathbb{Q}_p}$ of $\mathcal{Z}$ is contained in $\Phi^{-1} (\mathfrak{U})$. As usual, we identify the algebra of funcions $\mathcal{O}_{\mathcal{Z}_r} $ of the $\mathbb{Q}_p$-affinoid space $\mathcal{Z}_r$ with $\mathbb{E}\otimes \mathbb{Q}_p \langle \frac{T-(1-k)}{p^r}\rangle$. By restricting the morphism $\Phi$ to $\mathcal{Z}_r$, we get the morphism of $\mathbb{Q}_p$-affinoid spaces:
$$\Phi: \mathcal{Z}_r \rightarrow \mathfrak{U} $$
and as usual, we denote its associated morphism of $\mathbb{Q}_p$-affinoid algebras by $\Phi^{*} : \mathcal{O}_{\mathfrak{U}} \rightarrow \mathcal{O}_{\mathcal{Z}_r}$. \\
Now, observe that if we fix a $\bar{u}\in \mathfrak{U}$ with field of definition $\mathbb{L}_{\bar{u}}$, it induces a $\mathbb{Q}_p$-Banach spaces morphism given by the evaluation 
$\text{ev}_{\bar{u}} : \mathcal{O}_\mathfrak{U} \rightarrow \mathbb{L}_{\bar{u}}$. Consider now the finite, free $\mathcal{O}_{\mathfrak{U}}$-module $\mathbb{V}$ of rank 2 considered by Colmez. The ring homomorphism $\text{ev}_{\bar{u}}$ induces on $\mathbb{L}_{\bar{u}}$ a structure of $\mathcal{O}_\mathfrak{U}$-module, hence we define $\mathbb{V}(\bar{u}):=\mathbb{V}\otimes_{\mathcal{O}_\mathfrak{U}} \mathbb{L}_{\bar{u}}$.
By Chenevier's and Colmez's Theorem \ref{continuitya}, we have that $\mathbb{V}(\bar{u})=V(\delta_1 (\bar{u}), \delta_2 (\bar{u}))$ where $\bar{u}:=(\delta_1 (\bar{u}), \delta_2 (\bar{u}))\in\mathfrak{X}^2$.
In particular, we note that when $\bar{u}=u_{1-k}$, then $\mathbb{L}_{u_{1-k}}=\mathbb{E}$ and $\mathbb{V}(u_{1-k})=V(\mu_{a_p}, \delta_{k, a_p})$. Clearly the analogue statement holds for $k' \in \mathcal{Z}_r$ such that $k'-k\in(p-1)\mathbb{Z}_{\geq 0}$.\\
The idea is now to pull back the analytic family of representations given by Colmez in order to create a new analytic family parametrized by points in $\mathcal{Z}_r$ which has the advantage that will depend only on one parameter. The notion of  analytic family of representations parametrized  by an affinoid space is used in the sense of Berger and Colmez (see \cite{BC08}) but one could also have approached the problem from the point of view of analytic family of $(\varphi, \Gamma)$-modules over variations of the Robba rings in the sense of Bella\"iche (see \cite{Bel12}) considering the existence of a fully faithful functor which connects the two categories (see sec. 3 in \cite{Bel12}).\\
Since $\Phi: \mathcal{Z}_r\rightarrow \mathfrak{U}$ is an injective morphism of rigid analytic spaces, the induced map $\Phi^{*} : \mathcal{O}_\mathfrak{U}\rightarrow \mathcal{O}_{\mathcal{Z}_r}$ is given by the pull-back, i.e. $\Phi^{*} (f)=f\circ \Phi$ for all $f\in\mathcal{O}_\mathfrak{U}$. The ring homomorphism $\Phi^{*}$ gives to $\mathcal{O}_{\mathcal{Z}_r}$ the structure of $ \mathcal{O}_\mathfrak{U}$-module and so we can define $\mathbb{V}_{\mathbb{B}^{1,+}}:= \mathbb{V}\otimes_{\mathcal{O}_{\mathfrak{U}}} \mathcal{O}_{\mathcal{Z}_r}$, it is a finite, free $\mathcal{O}_{\mathcal{Z}_r}$-module of rank 2 with a continuous $G_{\mathbb{Q}_p}$-action (given by the action of the Galois group on $\mathbb{V}$). In particular, for all $s\in \mathcal{Z}_r$ we have by definition that $\mathbb{V}_{\mathcal{Z}_r}(s)=\mathbb{V}(\Phi(s))$.\\
Now, in order to deal wth reductions we first need to identify an integral family of lattices. First, we recall that there is a notion of integral model for affinoid algebras.
We define $\mathcal{O}_{\mathfrak{U}}^{0}:=\{g\in \mathcal{O}_{\mathfrak{U}}: \; |g|_{\text{sup}}\leq 1\}$. It is a model for $\mathcal{O}_{\mathfrak{U}}$, i.e. it is a $\mathbb{Z}_p$-subalgebra of $\mathcal{O}_{\mathfrak{U}}$
topologically of finite type (i.e. it is a quotient of $Z_p \langle x_1, \dots, x_n\rangle $ for some integer $n\geq 1$) and such that $\mathcal{O}_{\mathfrak{U}}^{0} [\frac{1}{p}]= \mathcal{O}_{\mathfrak{U}}$.\\
The existence inside $\mathbb{V}$ of an integral family follows from the following lemma:
\begin{lemma}
Let $G$ be a profinite group, let $A$ be a $\mathbb{Q}_p$-affinoid algebra and let $V$ be a finite, free $A$-module endowed with an $A$-linear, continuous action of $G$. Let $A^0$ be a model for $A$.
Then there exists a finite, free $A^0$-module $V_0$ inside $V$ such that $V_0\otimes_{A^0} A =V$ and which is stable under the continuous action of $G$. 
\end{lemma}

\begin{proof}
Let $W$ be a finite and free $A^0$-module such that $W\otimes_{A^0} A=V$. Since $A^0$ is open inside $A$, and since $V$ is a topological finite direct sum of copies of $A$ we have that $W$ is open inside $V$.
The action of $G$ can be represented by a continuous map $G \times W \rightarrow V$. Since $W$ is open inside $V$, then the subgroup $H_W \subset G$ stabilizing $W$ is an open subgroup of $G$. Since $G$ is profinite, we have 
that $H_W$ is of finite index. Let $\{h_i\}_i$ be a finite set of representatives for the left $H_W$-cosets in $G$. Hence, defining $V_0$ as $\sum_i h_i W$ we have that $V_0$ is a $G$-stable, finite and free $A^0$-module such that $V_0\otimes_{A^0} A=V$, this completes the proof.
\end{proof}
Applying the above result when $G=\text{Gal}(\overline{\mathbb{Q}}_p / \mathbb{Q}_p)$, $A=\mathcal{O}_\mathfrak{U}$,  $A^0=\mathcal{O}^0_\mathfrak{U}$ and $V=\mathbb{V}$ allow us to deduce that there exists $\mathbb{T}$ inside $\mathbb{V}$ finite, free $\mathcal{O}_{\mathfrak{U}}^{0}$-module of rank 2, which is $G_{\mathbb{Q}_p}$-stable and such that $\mathbb{T}\otimes \mathcal{O}_\mathfrak{U}=\mathbb{V}$ as 
$G_{\mathbb{Q}_p}$-modules.Then there exists $\mathbb{T}$ inside $\mathbb{V}$ finite, free $\mathcal{O}_{\mathfrak{U}}^{0}$-module of rank 2, which is $G_{\mathbb{Q}_p}$-stable. \\After defining $\mathbb{V}_{\mathcal{Z}_r}$ as $\mathbb{V}_{\mathcal{O}_\mathfrak{U}}\otimes_{\mathcal{O}_\mathfrak{U}} \mathcal{O}_{\mathcal{Z}_r}$, take the model $\mathcal{O}^{0}_{\mathcal{Z}_r}:=\{g\in \mathcal{O}_{\mathcal{Z}_r}: |g|_{\text{sup}}\leq 1 \}$. Since every morphism of affinoid algebras is in particular a contraction, we have that using the restriction, $\Phi: \mathcal{O}^{0}_{\mathfrak{U}} \twoheadrightarrow \mathcal{O}^{0}_{\mathcal{Z}_r}$. This allows to define $\mathbb{T}_{\mathcal{Z}_r}:=\mathbb{T}\otimes_{ \mathcal{O}^{0}_{\mathfrak{U}} } \mathcal{O}^{0}_{\mathcal{Z}_r}$.\\
The properties of $\mathbb{T}_{\mathbb{B}^{1,+}_r}$ are summarized in the following:

\begin{lemma}
The $\mathcal{O}^{0}_{\mathcal{Z}_r}$-module $\mathbb{T}_{\mathcal{Z}_r}$ is finite, free of rank 2 submodule of $\mathbb{V}_{\mathcal{Z}_r}$. It has a natural action of the Galois group $G_{\mathbb{Q}_p}$; in particular we have that:
$$\mathbb{T}_{\mathcal{Z}_r} \otimes_{\mathcal{O}^{0}_{\mathcal{Z}_r}} \mathcal{O}_{\mathcal{Z}_r}\cong \mathbb{V}_{\mathcal{Z}_r}$$
is an isomorphism of $G_{\mathbb{Q}_p}$-modules. 
\end{lemma}

\begin{proof}
The only claim that is not clear is the isomorphism of $G_{{\mathbb{Q}}_p}$-modules. By definition we have an isomorphism $\mathbb{T}\otimes_{\mathcal{O}^0_\mathfrak{U}} \mathcal{O}_\mathfrak{U}\cong \mathbb{V}$ of $G_{{\mathbb{Q}}_p}$-modules, hence tensorizing by $\mathcal{O}_{\mathcal{Z}_r}$ (considered as $\mathcal{O}_\mathfrak{U}$-module via $\Phi^{*}$) we get the isomorphism $(\mathbb{T}\otimes_{\mathcal{O}^0_\mathfrak{U}} \mathcal{O}_\mathfrak{U})\otimes_{\mathcal{O}_\mathfrak{U}} \mathcal{O}_{\mathcal{Z}_r} \cong \mathbb{V}_{\mathcal{Z}_r}$ of $G_{\mathbb{Q}_p}$-modules.\\
We have the following chain of isomorphism of $G_{\mathbb{Q}_p}$-modules:
$$
\begin{aligned}
&(\mathbb{T}\otimes_{\mathcal{O}^0_\mathfrak{U}} \mathcal{O}_\mathfrak{U})\otimes_{\mathcal{O}_\mathfrak{U}} \mathcal{O}_{\mathcal{Z}_r}\cong
\mathbb{T}\otimes_{\mathcal{O}^0_\mathfrak{U}} (\mathcal{O}_\mathfrak{U}\otimes_{\mathcal{O}_\mathfrak{U}} \mathcal{O}_{\mathcal{Z}_r})\cong 
\mathbb{T}\otimes_{\mathcal{O}^0_\mathfrak{U}}  \mathcal{O}_{\mathcal{Z}_r}\cong\\
&\mathbb{T}\otimes_{\mathcal{O}^0_\mathfrak{U}}  (\mathcal{O}^0_{\mathcal{Z}_r}\otimes_{\mathcal{O}^0_{\mathcal{Z}_r}} \mathcal{O}_{\mathcal{Z}_r})\cong
(\mathbb{T}\otimes_{\mathcal{O}^0_\mathfrak{U}}  \mathcal{O}^0_{\mathcal{Z}_r} )\otimes_{\mathcal{O}^0_{\mathcal{Z}_r}} \mathcal{O}_{\mathcal{Z}_r}\cong 
\mathbb{T}_{\mathcal{Z}_r}\otimes_{\mathcal{O}^0_{\mathcal{Z}_r}} \mathcal{O}_{\mathcal{Z}_r},
\end{aligned}
$$
where the isomorphisms are given by the associative property of tensor product for general modules (see Prop. 3.8, chap. 3 in \cite{Bou89}).
Hence, the claim follows. 
\end{proof}
Note that now for all $s\in\mathcal{Z}_r$, we have that $\mathbb{T}_{\mathcal{Z}_r} (s) \otimes \mathbb{L}_{u_s}=\mathbb{V}_{\mathcal{Z}_r}(s)$ and $\mathbb{T}_{\mathcal{Z}_r} (1-k)$ and $\mathbb{T}_{\mathcal{Z}_r} (1-k')$ are two $G_{\mathbb{Q}_p}$-stable $\mathcal{O}_\mathbb{E}$-lattices inside respectively $\mathbb{V}_{\mathcal{Z}_r} (1-k)$ and $\mathbb{V}_{\mathcal{Z}_r} (1-k')$.\\
Now, we are going to use the language of deformations theory to control the ``$p$-adic distance'' between $\mathbb{T}_{\mathcal{Z}_r} (1-k)$ and $\mathbb{T}_{\mathcal{Z}_r} (1-k')$.\\
Let $\rho : G_{\mathbb{Q}_p} \rightarrow \text{Gl}(\mathbb{T}_{\mathcal{Z}_r})\cong \text{Gl}_2 (\mathcal{O}_{\mathcal{Z}_r}^{0})$ the Galois representations attached to $\mathbb{T}_{\mathcal{Z}_r}$. 
For every $s\in\mathcal{Z}_r$, we define $\rho_s : G_{\mathbb{Q}_p}\rightarrow \text{Gl}(\mathbb{T}_{\mathcal{Z}_r} (s))$; as we have seen before, this representation correspond to a $G_{\mathbb{Q}_p}$-stable lattice inside the trianguline representation $V(\delta_1 (s), \delta_2 (s))$. In particular, we have that $\rho_{1-k}$ and $\rho_{1-k'}$ correspond to $G_{\mathbb{Q}_p}$-stable lattices inside respectively the representations $V(\mu_{a_p}, \delta_{k, a_p})$ and $V(\mu_{a_p}, \delta_{k', a_p})$. Moreover, as we have already seen, the representations $\rho_s$ can be obtained from $\rho$ via composition by the evaluation map at $s$, i.e. we have $\rho_s ={\text{ev}^{*}_{s}} \circ \rho$ where ${\text{ev}^{*}_{s}}$ is just notation for the induced map on $\text{Gl}_2$ from $\text{ev}_s$.\\
Now, for a fixed $m\in\mathbb{Z}_{\geq 1}$, we can consider the diagram:

$$
\begin{tikzcd}
       & &&\arrow[ld,"{{ev}_{1-k}^*}"']\text{Gl}(\mathbb{T}_{\mathcal{Z}_r} )\arrow[rd, "{{ev}_{1-k'}^*}"] &\\
        G_{\mathbb{Q}_p}   \arrow["{\rho}",rrru, bend left=15] \arrow["{\rho_{1-k}}", rr]& & \text{Gl}(\mathbb{T}_{\mathcal{Z}_r} (1-k))  \arrow[rd, "{\text{Pr}_m}"] & &  \text{Gl}(\mathbb{T}_{\mathcal{Z}_r} (1-k')) \arrow[ld, "{\text{Pr}_m}"']\\
        & & & \text{Gl}_2 (\mathcal{O}_\mathbb{E}/(p^m) & \\
        & &  & & 
    \end{tikzcd}
    $$
where $\text{Pr}_m$ denotes the induced homomorphism on $\text{Gl}_2$ from the natural projection $\mathcal{O}_\mathbb{E}\twoheadrightarrow \mathcal{O}_\mathbb{E} /(p^m)$.\\
It is clear that the above diamond in the diagram commutes if and only if for all $f\in\mathcal{O}^{0}_{\mathcal{Z}_r}$ we have $f(1-k)-f(1-k') \in (p^m)$ inside $\mathcal{O}_\mathbb{E}$.\\
Hence, we reduced the claim to prove that there exists a positive integer $n=n(k, a_p, m)\geq 1$ such that if $k'-k \in p^n (p-1)\mathbb{Z}_{\geq 0}$ then $|\text{ev}_{1-k'} (f)-\text{ev}_{1-k} (f) |=|f(1-k')-f(1-k)|\leq p^{-m}$.\\
This follows from the general following lemma in the theory of affinoid algebras:

\begin{lemma}
Let $r\geq 0$ be an integer. For every $g\in \overline{\mathbb{Q}}_p \langle \frac{T}{p^r}\rangle:=\{\sum_n a_n T^n  :  a_n p^{rn} \rightarrow 0 \text{ as }n\rightarrow \infty\}$ and for any $x, y \in \mathbb{B}^{1} (0, p^{-r})^{+}$ we have 

$$|g(x)-g(y)| \leq p^{r} |g|_{r} |x-y|.$$ 
Here $|\cdot|_r$ denotes the norm on $\overline{\mathbb{Q}}_p \langle \frac{T}{p^r}\rangle$ given by $|\sum a_n T^n|:=\text{max } |a_n p^{rn}|$ and $|\cdot|$ denotes the usual norm on $\overline{\mathbb{Q}}_p \langle T\rangle$.
\end{lemma}

\begin{proof}
Consider the map $\alpha: \overline{\mathbb{Q}}_p \langle \frac{T}{p^r}\rangle \rightarrow \overline{\mathbb{Q}}_p \langle T\rangle$ sending $T$ to $p^r T$; it is an isometric isomorphism with respect to the norms $|\cdot|_r$ and $|\cdot|$ respectively. Denote by $\alpha^{*}$ the induced map on maximal spectra, i.e. $\alpha^{*} : \mathbb{B}^{1} (0, 1)^{+}\rightarrow \mathbb{B}^{1} (0, p^{-r})^{+}$ sending $z$ to $p^r z$; it is bijective. Let $g\in\overline{\mathbb{Q}}_p \langle \frac{T}{p^r}\rangle$ such that $\alpha(g)=f$.\\
It is a classical result in the theory of Tate's algebras (see Prop. 7.2.1.1 in \cite{BGR84}) that for any $f\in \overline{\mathbb{Q}}_p \langle T\rangle$ and for any $\tilde x, \tilde y \in \mathbb{B}^{1} (0, 1)^{+}$:
$$
\begin{aligned}
|f(\tilde x)-f(\tilde y)|\leq |f| |\tilde x-\tilde y|.
\end{aligned}
$$
Since $\alpha$ is an isometry and defining $\alpha^{*}(\tilde x)= x$ and $\alpha^{*} ( \tilde y)=y$, this is equivalent to say:

$$
\begin{aligned}
&|\alpha(g)(\tilde x)-\alpha(g)(\tilde y)|\leq |f| \cdot |\tilde x-\tilde y|\\
\Longleftrightarrow &|g(\alpha^{*}(\tilde x))-g(\alpha^{*}(\tilde y))|\leq |\alpha(g)|\cdot |\tilde x-\tilde y|\\
\Longleftrightarrow &|g(x)-g(y)|\leq |g|_r \cdot |\tilde x- \tilde y|=|g|_r \cdot \Big|\frac{x}{p^r}- \frac{y}{p^r}\Big|=p^r |g|_r |x-y|.
\end{aligned}
$$
Note that we used the fact that $\alpha(g)(\tilde x)=g(\alpha^{*}(\tilde x))$ (and the same for $y$) which is a standard property of affinoid maps (see Lemma 7.1.4.2 in \cite{BGR84}).
\end{proof}
Finally, we can complete the proof of theorem \ref{weight}. Indeed, we have that the model $\mathcal{O}^{o}_{\mathcal{Z}_r}$ is isomorphc to $\mathbb{E} \langle \frac{T}{p^r}\rangle^{0} :=\{g \in \mathbb{E} \langle \frac{T}{p^r}\rangle : |g|_{\text{sup}}=|g|_r \leq 1\}$. Hence, for all $g\in \mathbb{E} \langle \frac{T}{p^r}\rangle^{0} $ we have:
$|g(x)-g(y)|\leq p^r |x-y| \;\;\; \text{ for all }\; x, y\in \overline{\mathbb{Z}}_p$ representing the corresponding maximal ideals in $\mathcal{Z}_r$. \\
For any fixed positive integer $m$ such that the hypothesis of the theorem holds, there exists a positive integer $n$, namely $n=m+r$, such that the representations $\mathbb{T}_{\mathcal{Z}_r}(1-k)$ and $\mathbb{T}_{\mathcal{Z}_r}(1-k')$ are congruent modulo $p^m$. By the definition of $\mathbb{T}_{\mathcal{Z}_r}$ we deduce that the same is true for $\mathbb{T}(u_{1-k})$ and $\mathbb{T}(u_{1-k'})$. This completes the proof of the theorem
\printbibliography[heading=bibintoc,title={References}]

\end{document}